\documentclass{amsart}
\usepackage{amsmath,amssymb,epic,graphicx,mathrsfs,enumerate,xcolor}
\usepackage{amsfonts}
\usepackage{amsthm}
\usepackage{amssymb}
\usepackage{latexsym}
\usepackage{longtable}
\usepackage{epsfig}
\usepackage{amsmath}
\usepackage{hhline}
\usepackage{float}
\restylefloat{table}

\input xy
\xyoption{all}

\newtheorem{teor}{Theorem}
\newtheorem{cor}{Corollary}
\newtheorem{prop}{Proposition}

\newtheorem{lemma}{Lemma}

\DeclareMathOperator{\Aut}{Aut}

\DeclareMathOperator{\soc}{soc}

\begin{document}

\title[]{On minimal coverings and pairwise generation of some primitive groups of wreath product type}

\author{Julia Almeida} 
\address{Departamento de Matem\'atica, Universidade de Bras\'ilia, Campus 
Universit\'ario Darcy Ribeiro, Bras\'ilia-DF, 70910-900, Brazil}
\email{julia\_aredes\_almeida@hotmail.com}

\author{Martino Garonzi}
\address{Departamento de Matem\'atica, Universidade de Bras\'ilia, Campus 
Universit\'ario Darcy Ribeiro, Bras\'ilia-DF, 70910-900, Brazil \newline
ORCID: https://orcid.org/0000-0003-0041-3131}
\email{mgaronzi@gmail.com}

\thanks{The authors acknowledge the support of Conselho Nacional 
de Desenvolvimento Cient\'ifico e Tecnol\'ogico (CNPq), PhD fellowship and Universal
- Grant number 402934/2021-0.}

\date{}



\begin{abstract}
The covering number of a finite group $G$, denoted $\sigma(G)$, is the smallest positive integer $k$ such that $G$ is a union of $k$ proper subgroups. We calculate $\sigma(G)$ for a family of primitive groups $G$ with a unique minimal normal subgroup $N$, isomorphic to $A_n^m$ with $n$ divisible by $6$ and $G/N$ cyclic. This is a generalization of a result of E. Swartz concerning the symmetric groups. We also prove an asymptotic result concerning pairwise generation.
\end{abstract}

\maketitle

\section{Introduction}

In this paper, all groups are assumed to be finite. A covering of a group $G$ is a family of proper subgroups of $G$ whose union is $G$ and the covering number of $G$, denoted $\sigma(G)$, is the smallest size of a covering of $G$. This interesting invariant was introduced by J. Cohn in \cite{cohn} and it was later studied by many authors. In this paper, we focus our attention on a family of groups closely related to symmetric groups, so let us shortly recall what is known about $\sigma(S_n)$. It is easy to prove that $\sigma(S_3)=\sigma(S_4)=4$. In his paper, J. Cohn proved, among other things, that $\sigma(S_5)=16$. A. Abdollahi et al \cite{s6} proved that $\sigma(S_6)=13$. A. Mar\'oti \cite{marotisym} proved that $\sigma(S_n)=2^{n-1}$ for $n \geqslant 7$ odd and $n \neq 9$, L.-C. Kappe, D.  Nikolova-Popova and E. Swartz \cite{kappe} proved that $\sigma(S_8)=64$, $\sigma(S_9)=256$, $\sigma(S_{10})=221$, $\sigma(S_{12})=761$, R. Oppenheim and E. Swartz \cite{s14} proved that $\sigma(S_{14})=3096$ and E. Swartz \cite{Eric} proved that $\sigma(S_{18})=36773$ and gave a precise formula for $\sigma(S_n)$ when $n \geqslant 30$ is divisible by $6$, which coincides with the formula in our Theorem \ref{main_1} below setting $m=1$.

If $G$ is $2$-generated, the generating graph of $G$ is the simple graph whose vertices are the elements of $G$ and two vertices are connected by an edge if together they generate $G$. A clique of a simple graph is a complete subgraph and its clique number is the maximal size of a clique. We denote by $\omega(G)$ the clique number of the generating graph of $G$, in other words $\omega(G)$ is the maximal size of a subset $S$ of $G$ with the property that $\langle x,y \rangle = G$ whenever $x,y \in S$ and $x \neq y$. Since any proper subgroup of $G$ can contain at most one element of such a set $S$, we have $\omega(G) \leqslant \sigma(G)$. It is very natural to ask whether equality occurs for some families of groups, at least asymptotically. S. Blackburn \cite{Blackburn} proved that $\sigma(S_n)=\omega(S_n)$ if $n$ is odd and sufficiently large, and later L. Stringer \cite{Stringer} proved that $\omega(S_n)=\sigma(S_n)$ for all odd $n$ different from $9$ and from $15$, and that $\sigma(S_9) \neq \omega(S_9)$. It is not known wheter $\omega(S_{15})$ equals $\sigma(S_{15})$ or not. In a joint work with F. Fumagalli and A. Mar\'oti \cite{MFA}, the second author proved that $\omega(S_n)/\sigma(S_n)$ tends to $1$ when $n$ is even and tends to infinity. A. Lucchini and A. Mar\'oti \cite{lm} proved that $\sigma(G)=\omega(G)$ if $G$ is a solvable group of Fitting length at most $2$.

A group $G$ is called primitive if it admits a maximal subgroup with trivial normal core. The study of $\sigma(G)$ and $\omega(G)$ for primitive groups is crucial for the understanding of the general behaviour of these invariants. Indeed, for a general $G$, we have $\sigma(G)=\sigma(G/\Phi(G))$ and, if $G$ is $2$-generated, $\omega(G)=\omega(G/\Phi(G))$, where $\Phi(G)$ denotes the Frattini subgroup of $G$, and it is easy to see that $G/\Phi(G)$ is a subdirect product of primitive groups. Denote by $d(G)$ the minimal size of a set of generators of $G$. If $G$ is noncyclic and has a unique minimal normal subgroup, call it $N$, then \cite[Theorem 1.1]{Lucchini} implies that $d(G)=\max\{2,d(G/N)\}$. Moreover, in this case, if $\Phi(G)=\{1\}$ then $G$ is primitive. Since we are interested in $2$-generated groups, the first case to consider is the one in which $G/N$ is cyclic.

 Our objective in this paper is to continue the work started in \cite{MA} and \cite{Martino} concerning the covering number of specific families of primitive groups with nonabelian socle. Recall that the socle of a group $G$ is the subgroup generated by the minimal normal subgroups of $G$, and if $G$ is primitive with nonabelian socle $N$, then $N$ is the unique minimal normal subgroup of $G$ and it is a direct power $T^m$ of a nonabelian simple group $T$. Assume $G/N$ is cyclic and that $T$ is isomorphic to an alternating group $A_n$ with $n \geqslant 5$. If $T_1$ denotes the first direct factor of $N$, then the structure of $G$ is determined by the almost-simple group $X=N_G(T_1)/C_G(T_1)$, which has socle isomorphic to $T$, so it can be one of $A_n$ and $S_n$, if $n \neq 6$. If $X \cong A_n$, then $G$ is the wreath product $A_n \wr C_m$, where $C_m$ acts as an $m$-cycle. In this paper, we are interested in the case $X \cong S_n$. In this case, the structure of $G$ is as follows. 

Let $G = G_{n,m}$ be the semidirect product $A_n^m \rtimes \langle \gamma \rangle$ where $\gamma = (1, \ldots, 1, \tau) \delta \in S_n \wr S_m$, with $\tau = (1 \; 2)$ and $\delta = (1 \ldots  m)$. If $x_1,\ldots,x_m \in A_n$, we have
$$(x_1, \ldots, x_m )^{\gamma} = ({x_m}^{\tau}, x_1, \ldots, x_{m-1}).$$
In this paper, we establish the following result, generalizing the main result of \cite{Eric} about $\sigma(S_n)$, which corresponds to the case $m = 1$.

\begin{teor} \label{main_1}
Let $G = G_{n,m}$, for $n \geqslant 30$ divisible by $6$ and $m \geqslant 2$. Denote by $\alpha(x)$ the number of distinct prime factors of the positive integer $x$. Then
$$\sigma(G) = \alpha(2m) + \left( \frac{1}{2} \binom{n}{n/2} \right)^m + \sum_{i=1}^{n/3-1} \binom{n}{i}^m.$$
Moreover, $G$ has a unique minimal covering consisting of maximal subgroups.
\end{teor}

Let $G$ be the group in the above statement. Denote by $d(G)$ the minimal number of elements needed to generate $G$ and let $N$ be the socle of $G$. Since $N$ is the unique minimal normal subgroup of $G$ and $G/N$ is cyclic, the main Theorem of \cite{Lucchini} implies that $d(G)=2$. So it makes sense to consider $\omega(G)$. Two explicit generators of $G$ can be constructed as follows: let $x_1,x_2 \in A_n$ be such that $\langle x_1 \tau, x_2 \tau \rangle = S_n$. Then $\langle \alpha_1,\alpha_2 \rangle = G$ where $\alpha_i = (x_i,1,\ldots,1)\gamma$ for $i=1,2$. Our second result is the following.

\begin{teor} \label{main_2}
Set $G:=G_{n,m}$. For fixed $m \geqslant 2$, $\omega(G)$ is asymptotically equal to $\left( \frac{1}{2} \binom{n}{n/2} \right)^m$ for $n \to \infty$, $n$ even, and $\omega(G)/\sigma(G)$ tends to $1$ as $n \to \infty$, $n$ even.
\end{teor}

Note that the second statement of the theorem follows from the first one using \cite[Theorem 1 (3)]{Martino}. It is interesting to ask whether $\omega(G)/\sigma(G)$ tends to $1$ when $G$ is a $2$-generated primitive group and $|G| \to \infty$. For an interesting example of a family of groups $G$ for which $\omega(G)/\sigma(G)$ tends to $0$, see \cite[Theorem 1.2]{lm}.

\section{Proof of Theorem \ref{main_1}} \label{section_main_1}

In this section, we prove Theorem \ref{main_1}. The strategy is to apply Lemma 3.1 of \cite{Eric}. Let us explain this here for the convenience of the reader. Assume $G$ is a finite group whose conjugacy classes of maximal subgroups are indexed by a set $I_G$. For $j \in I_G$, let $\mathcal{M}_j$ be the corresponding conjugacy class of maximal subgroups of $G$. Let $J$ be a subset of $I_G$ and assume that $\mathcal{C} = \bigcup_{j \in J} \mathcal{M}_j$ is a covering of $G$. Let $\Pi$ be a subset of $G$ closed under conjugation and denote by $\Pi_j$ the subset of $\Pi$ covered by the conjugacy class $\mathcal{M}_j$, so that $\Pi_j$ is closed under conjugation if $\Pi$ is closed under conjugation. If this is the case, and $M,M'$ are conjugate maximal subgroups of $G$ and $j \in J$, then $|M \cap \Pi|=|M' \cap \Pi|$ and $|M \cap \Pi_j| = |M' \cap \Pi_j|$. For a maximal subgroup $M$ of $G$ such that $M \not \in \mathcal{C}$, let
$$d(M) := \sum_{j \in J} \frac{|M \cap \Pi_j|}{|M_j \cap \Pi_j|}$$
where $M_j$ is any fixed member of $\mathcal{M}_j$. 
\begin{lemma}[Lemma 3.1 of \cite{Eric}] \label{EricLemma}
Assume that the following conditions hold for the covering $\mathcal{C}$ defined above.
\begin{enumerate}
    \item $x^g \in \Pi$, for all $x \in \Pi$ and $g \in G$, i.e. $\Pi$ is closed under conjugation.
    \item For every $\pi \in \Pi$, there is a unique member of $\mathcal{C}$ containing $\pi$.
    \item $d(H) < 1$ for every maximal subgroup $H$ of $G$ not in $\mathcal{C}$.
\end{enumerate}
Then $\mathcal{C}$ is a minimal covering of $G$, meaning that $\sigma(G)=|\mathcal{C}|$. Moreover $\mathcal{C}$ is the unique minimal covering of $G$ consisting of maximal subgroups.
\end{lemma}

Let $G$ be the group defined in the statement of Theorem \ref{main_1}. We will construct $J$, $\Pi$ and $\mathcal{C}$ to apply Lemma \ref{EricLemma} to the group $G$. Using the language of \cite{Spagnoli}, which we will often refer to, $G$ is a primitive group of type 2, meaning that $G$ has a core-free maximal subgroup and it admits precisely one minimal normal subgroup, which is nonabelian: its socle, $N=A_n^m$. 

Let $I = \{ -1, 1, 2, \ldots, n/3 - 1 \}$. As in \cite{Eric}, we define collections $B_i$, $i \in I$, as follows. For simplicity of notation, let us denote by $[a_1,\ldots,a_k]$ the conjugacy class of $S_n$ corresponding to the elements of cycle structure $(a_1,\ldots,a_k)$, where of course the $a_i$'s are positive integers which sum to $n$. Set 
$$\begin{array}{ll}
B_{-1} := [n], \\
B_1 := [1, n/2 - 2, n/2 + 1] \\
B_2 := [2,n/2-1,n/2-1] & \mbox{if } n/2 \mbox{ is even,} \\
B_2 := [2,n/2-4,n/2+2] & \mbox{if } n/2 \mbox{ is odd,} \\
B_i := [i, (n-i-1)/2, (n-i+1)/2] & \mbox{if } i \mbox{ is odd, } 3 \leqslant i < n/3, \\
B_i := [i,(n-i)/2,(n-i)/2] & \mbox{if } i \mbox{ is even, } (n-i)/2 \mbox{ is odd, } 4 \leqslant i < n/3, \\
B_i := [i,(n-i)/2-1,(n-i)/2+1] & \mbox{if } i \mbox{ is even, } (n-i)/2 \mbox{ is even, } 4 \leqslant i < n/3.
\end{array}$$ 

Note that $B_i \cap A_n = \varnothing$ for all $i \in I$. We define the set $\Pi_i$ for all $i \in I$ as follows:
$$\Pi_i = \{ (x_1, x_2, \ldots , x_m) \gamma \in G\ :\ x_1 x_2 \ldots x_m \tau \in B_i \}.$$ 
Note that the sets $\Pi_i$ are pairwise disjoint as are the sets $B_i$.

Note that we did not define $\Pi_0$ yet. Rather than defining a unique $\Pi_0$, we will define several sets, which we will call $\Pi_{0,r}$, for every prime $r$ dividing $m$. For such $r$, let $D_1$ be the conjugacy class of $(n-2)$-cycles in $S_n$, and let $D_i$ be the conjugacy class of $n$-cycles in $S_n$ for $i=2,\ldots,r$. Let $\nu := (1 \ldots r)$. For all $\sigma \in \langle \nu \rangle$, let
$$\Pi_{0, r, \sigma} := \{ (x_1, \ldots , x_m) \gamma^r \in G\ :\ x_i x_{i+r} x_{i+2r} \ldots x_{i+m-r} \tau \in D_{\sigma(i)}\ \forall i =1,\ldots,r \}.$$
Assume that either $m$ is even or $r \neq 2$. We define
$$\Pi_{0,r} := \bigcup_{\sigma \in \langle \nu \rangle} \Pi_{0,r,\sigma}.$$
This is a disjoint union. Indeed, let $\sigma_1, \sigma_2 \in \langle \nu \rangle$ with $\sigma_1 \neq \sigma_2$ and assume by contradiction that there exists $(x_1, \ldots , x_m ) \gamma^r \in \Pi_{0, r, \sigma_1} \cap \Pi_{0, r, \sigma_2}$. Since $\sigma_1 \neq \sigma_2$ and $r$ is a prime, there exists $i \in \{1,\ldots,r\}$ such that $\sigma_1(i)=1$ and $j=\sigma_2(i) \neq 1$. Since $x_i x_{i + r} \ldots x_{i + m - r} \tau$ belongs to $D_{\sigma_1(i)} \cap D_{\sigma_2(i)} = D_1 \cap D_j$, we deduce that $D_1=D_j$, a contradiction. It follows that $$|\Pi_{0,r}| = r \cdot |A_n|^{m-r} \cdot \prod_{i=1}^r |D_i|.$$

Assume now that $m$ is odd. We will define $\Pi_{0,2}$. Consider the conjugacy class $C$ of $S_n$ consisting of the elements of cycle structure $(p,n-p)$ where $p$ is a fixed prime number such that $n/3 < p < 2n/3$. Note that $p$ exists by Bertrand's postulate. In this case, we define
$$ \Pi_{0,2} := \{ (x_1, \ldots , x_m) \gamma^2 \in G\ :\ x_1 x_3 \ldots x_m \tau \cdot x_2 x_4 \ldots x_{m-1} \tau \in C \} .$$
Note that 
$$|\Pi_{0,2}| = |A_n|^{m-1} \cdot |C|.$$
Let $J$ be the set of indices consisting of the elements of $I$ and the pairs $(0,r)$ where $r$ is a prime divisor of $2m$ and set $\Pi := \bigcup_{j \in J} \Pi_j$. The following proposition shows that every $\Pi_j$ is closed under conjugation, proving that condition (1) of Lemma \ref{EricLemma} holds. 

\begin{prop} \label{classedeconj_i}
For all $i \in I$, the sets $\Pi_i$, $i \in I$, and $\Pi_{0,r}$, $r$ any prime divisor of $2m$, are closed under conjugation.
\end{prop}
\begin{proof}
Fix $i \in I$. If $(x_1, \ldots , x_m) \gamma \in \Pi_i$, the element
\begin{align*}
((x_1, \ldots , x_m) \gamma)^{\gamma} = (\tau x_m \tau, x_1, \ldots , x_{m-1}) \cdot \gamma
\end{align*}
belongs to $\Pi_i$ because 
$$\tau x_m \tau \cdot x_1 \ldots x_{m-1} \cdot \tau = (x_1 \ldots x_m \tau)^{\tau x_m^{-1} \tau} \in B_i,$$
and if $(y_1, \ldots , y_m) \in A_n^m$, the element
$$ ((x_1, \ldots , x_m) \gamma)^{(y_1, \ldots , y_m)} =  (y_1^{-1} x_1 y_2, y_2^{-1} x_2 y_3 , \ldots, y_{m-1}^{-1} x_{m-1} y_m , y_m^{-1} x_m \tau y_1 \tau) \gamma$$
belongs to $\Pi_i$ because
$$y_1^{-1} x_1 y_2 \cdot y_2^{-1} x_2 y_3 \cdot \ldots \cdot y_{m-1}^{-1} x_{m-1} y_m \cdot y_m^{-1} x_m \tau y_1 \tau \cdot \tau = (x_1 \ldots x_m \tau)^{y_1} \in B_i .$$
Since $G$ is generated by $A_n^m$ and $\gamma$, this proves that $\Pi_i$ is closed under conjugation.

We now prove that $\Pi_{0,r}$ is closed under conjugation. The following argument can be applied to the case $r=2$ when $m$ is odd, so we will assume that either $m$ is even or $r \neq 2$. Let $(x_1,\ldots,x_m) \gamma^r \in \Pi_{0,r,\sigma}$. Note that
\begin{align*}
((x_1, \ldots , x_m) \gamma^r)^{\gamma} & =
(\tau x_m \tau, x_1, \ldots , x_{m-1}) \gamma^r
\end{align*}
and we have the following.
$$\begin{array}{l}
\tau x_m \tau x_r x_{2r} \ldots x_{m-r} \tau = (x_r x_{2r} \ldots x_{m-r} x_m \tau)^{\tau x_m^{-1} \tau} \in D_{\sigma(r)} = D_{\sigma \nu^{-1} (1)} \\
x_i x_{i+r} x_{i+2r} \ldots x_{i+m-r} \tau \in D_{\sigma(i)} = D_{\sigma \nu^{-1} (i+1)} \hspace{.5cm} \forall i=1,\ldots,r-1.
\end{array}$$
 
It follows that $((x_1,\ldots,x_m)\gamma^r)^{\gamma} \in \Pi_{0,r,\sigma \nu^{-1}} \subseteq \Pi_{0,r}$. 
 
For $(y_1, \ldots , y_m) \in A_n^m$ we have that $((x_1, \ldots , x_m) \gamma^r)^{(y_1, \ldots , y_m)}$ equals
\begin{align*}
(y_1^{-1} x_1 y_{r+1}, \ldots , y_{m-r}^{-1} x_{m-r} y_m , y_{m-r+1}^{-1} x_{m-r+1} \tau y_1 \tau, \ldots , y_m^{-1} x_m \tau y_r \tau ) \cdot \gamma^{r}
\end{align*}

Moreover, if $1 \leqslant i \leqslant r$,
$$y_i^{-1} x_i y_{r+i} \cdot y_{r+i}^{-1} x_{r+i} y_{2r+i} \cdot \ldots \cdot y_{m-r+i}^{-1} x_{m-r+i} \tau y_i \tau \cdot \tau 
= (x_i x_{i+r} x_{i+2r} \ldots x_{i+m-r} \tau)^{y_i}$$
belongs to $D_{\sigma(i)}$. This implies that $\Pi_{0,r}$ is closed under conjugation.
\end{proof}

For $i \in I$, $i \neq -1$, define $\mathcal{E}_i$ to be the set of maximal intransitive subgroups of $A_n$ whose orbits have size $i$ and $n-i$ and let $\mathcal{E}_{-1}$ be the set of maximal imprimitive subgroups of $A_n$ with $2$ blocks. Let $\mathcal{F}_i := \{N_{S_n}(M)\ :\ M \in \mathcal{E}_i\}$ for all $i \in I$, $\mathcal{E} := \bigcup_{i \in I} \mathcal{E}_i$ and $\mathcal{F} := \bigcup_{i \in I} \mathcal{F}_i$. Note that $\{A_n\} \cup \mathcal{F}$ is a covering of $S_n$, as observed in \cite{Eric}. By \cite[Lemma 5.2]{Eric}, for $n \equiv 0 \mod 6$, $n \geqslant 30$ and $i \in I$, the only subgroups in $\mathcal{F}$ that contain elements of $B_i$ are the ones belonging to $\mathcal{F}_i$, so that the elements of $\bigcup_{i \in I} B_i$ are partitioned by the subgroups in $\mathcal{F}$. Moreover $\mathcal{E}_i$ and $\mathcal{F}_i$ are conjugacy classes of subgroups of $S_n$ for all $i \in I$.

For $i \in I$, define $\mathcal{M}_i$ to be the set consisting of the subgroups of $G$ of the following type: $H = N_G(M^{a_1} \times \ldots \times M^{a_m})$ where $a_1,\ldots,a_m \in A_n$, $M \in \mathcal{E}_i$ and $N_{S_n}(M) \cap B_i \neq \varnothing$. By \cite[Proposition 1.1.44]{Spagnoli} and \cite{LPS}, $H$ is a maximal subgroup of $G$ supplementing the socle $N=A_n^m$ of $G$, moreover $H \cap N$ is conjugate to $M^m$ in $N$. It follows that $|H|=2m \cdot |M|^m$ and $H$ has $|G:H|=|A_n:M|^m$ conjugates in $G$. For every prime divisor $r$ of $2m$, set $\mathcal{M}_{0,r} = \{M_{0,r}\}$ where $M_{0,r}=A_n^m \rtimes \langle \gamma^r \rangle$ is a normal subgroup of $G$ of index $r$. Let $\mathcal{C}$ be the union of all the $\mathcal{M}_j$, for $j \in J$. The size of $\mathcal{C}$ equals the claimed value for $\sigma(G)$ in the statement of Theorem \ref{main_1}.

\begin{prop}
    $\mathcal{C}$ is a covering of $G$.
\end{prop}

\begin{proof}
Let $g = (x_1,\ldots,x_m)\gamma^k \in G$ where $x_i \in A_n$ for all $i$. If $(k, m) \neq 1$ then $g$ belongs to one of the $\alpha(2m)$ subgroups of $G$ containing the socle, now suppose that $(k, m) = 1$. Since $\langle g \rangle = \langle g^t \rangle$ if $t$ is coprime to the order of $g$, we can assume that $k=1$. Since $\mathcal{F}$ is a covering of $S_n$, there exists $M \in \mathcal{E}$ such that the odd permutation $x_1 \ldots x_m \tau$ belongs to $N_{S_n}(M)$ and
$$H = N_G(M \times M^{x_1} \times M^{x_1x_2} \times \ldots \times M^{x_1 \ldots x_{m-1}})$$
is a member of $\mathcal{C}$ containing $g$. 
\end{proof}

\begin{prop} \label{M_i}
The sets $\mathcal{M}_i$, $i \in I$, are conjugacy classes of subgroups of $G$.
\end{prop}

\begin{proof}
A given subgroup in $\mathcal{M}_i$ is $A_n^m$-conjugate to $H = N_G(M^m)$ where $M \in \mathcal{E}_i$, so since every member of $\mathcal{E}_i$ is an $A_n$-conjugate of $M$, being $N_{S_n}(M)A_n=S_n$, we only need to show that $H^{\gamma} = N_G(M^{\tau} \times M^{m-1})$ is $A_n^m$-conjugate to $H$. This follows from the fact that $M^{\tau}$ is $A_n$-conjugate to $M$, being $N_{S_n}(M)A_n = S_n$.
\end{proof}

We will now describe the maximal subgroups of $G$. A reference for the following discussion is \cite{Spagnoli}. The maximal subgroups of $G$ containing the socle $N=\soc(G)$ are the $M_{0,r}$ where $r$ is any prime divisor of $2m$. Let $U$ be a maximal subgroup of $G$ not containing $N$, so that $UN=G$. Observe that $U \cap N \neq \{1\}$. Indeed, if by contradiction $U \cap N = \{1\}$, then $U \cong G/N$ would be cyclic, generated by an element $u$, therefore the only proper subgroup of $G$ containing $u$ would be $U$, and this contradicts the fact that $\mathcal{C}$ is a covering of $G$ whose members are not cyclic. Then $U$ can be of one of the following two types. The first type is $U = N_G(U \cap N)$ where $U \cap N = M \times M^{a_2} \times \ldots \times M^{a_m}$, $a_2,\ldots,a_m \in A_n$ and $M$ is the intersection between $A_n$ and a maximal subgroup of $S_n$. In this first case, $U$ is called a maximal subgroup of product type 
(see \cite[Proposition 1.1.44, Definition 1.1.45]{Spagnoli}). The second type consists of maximal subgroups of diagonal type (see \cite[Proposition 1.1.55]{Spagnoli}). Fix a partition $\{P_1,\ldots,P_k\}$ of $\Omega=\{1,\ldots,m\}$ and write $P_i=\{a_{ij}\ :\ j=1,\ldots,r_i\}$. Given a collection of automorphisms $\varphi_{ij}$ of $A_n$, with $i=1,\ldots,k$ and $j=2,\ldots,r_i$, let $\Delta_{\varphi}$ be the set of $m$-tuples $(x_1,\ldots,x_m) \in A_n^m$ with the property that $x_{a_{ij}} = x_{a_{i1}}^{\varphi_{i,j}}$ for all $i,j$. Then we set $U$ to be the normalizer of $\Delta_{\varphi}$ in $G$. If $U$ supplements $N$, there is in $U$ an element $u = (y_1,\ldots,y_m) \delta$ where each $y_i$ belongs to $S_n$ and it is easy to see that the partition $P$ is stabilized by $\delta$. If $U$ is a maximal subgroup of $G$, then we may assume that $P$ is minimal, with respect to the relation of refinement, among the nontrivial partitions stabilized by $\delta$, in other words 
$$\Delta_{\varphi} = \{ (y_1, \ldots, y_{m/t}, y_1^{\varphi_{1,2}} , \ldots , y_{m/t}^{\varphi_{m/t, 2}} , \ldots , y_1^{\varphi_{1, t}} , \ldots , y_{m/t}^{\varphi_{m/t, t}} ) :\ y_1, \ldots , y_{m/t} \in A_n\}$$
where $t$ is a prime divisor of $m$, $\varphi_{i, j}$ is an automorphism of $A_n$ for $1 \leqslant i \leqslant m/t$, $2 \leqslant j \leqslant t$, and the matrix $(\varphi_{i, j})_{i,j}$ is denoted by $\varphi$. If $U$ is a maximal subgroup of $G$, supplementing the socle $N$, and of the form $N_G(\Delta_{\varphi})$ with $\varphi$ as above then $U$ is called a maximal subgroup of diagonal type. If this is the case, then $U \cap N = \Delta_{\varphi}$.

In the following discussion, we fix a subgroup $M$ of $A_n$ such that $N_G(M^m)$ is a maximal subgroup of $G$ which supplements the socle $N=\soc(G)$, in other words $N_G(M^m)N=G$.

\begin{lemma} \label{ncontido}
$N_{S_n}(M)A_n=S_n$, in particular $N_{S_n}(M) \nsubseteq A_n$.
\end{lemma}

\begin{proof}
Let $\alpha \in S_n$. If $\alpha \in A_n$ then $\alpha \in N_{S_n}(M)A_n$, so now assume that $\alpha \not \in A_n$. Then $\alpha \tau \in A_n$ and so $(\alpha,\ldots,\alpha) = (\alpha \tau, \ldots , \alpha \tau)(\tau , \ldots , \tau) \in G$, being $(\tau,\ldots,\tau)=\gamma^m$. By assumption, we can write $(\alpha, \ldots , \alpha) = n h$, where $n = (a_1, \ldots , a_m) \in A_n^m$ and $h = (b_1, \ldots , b_m) \gamma^k \in N_G(M^m)$. It follows that $k=0$ and hence $b_i \in N_{S_n}(M)$ for all $i=1,\ldots,m$. Therefore $\alpha = a_1b_1 \in A_n N_{S_n}(M)$. 
\end{proof}

\begin{lemma} \label{pi_0}
Let $g \in G$ and let $r$ be a prime divisor of $2m$. If either $m$ is even or $r \neq 2$, then
$$|N_G(M^m)^g \cap \Pi_{0,r} |  = r \cdot \displaystyle \left( \frac{1}{2} \; | N_{S_n}(M)| \right)^{m-r}  \cdot \prod_{i = 1}^r |D_i  \cap N_{S_n}(M) |.$$ 
If $m$ is odd, then 
$$|N_G(M^m)^g \cap \Pi_{0,2}| = \left( \frac{1}{2} \; |N_{S_n}(M)| \right)^{m-1} \cdot |C \cap N_{S_n}(M)|.$$
\end{lemma}

\begin{proof}
Since $\Pi_{0,r}$ is closed under conjugation, the size of $N_G(M^m)^g \cap \Pi_{0,r}$ equals the size of $N_G(M^m) \cap \Pi_{0,r}$, therefore we may assume that $g=1$. Assume first that either $m$ is even or $r \neq 2$. We will compute $|N_G(M^m) \cap \Pi_{0,r,\sigma}|$ for each $\sigma \in \langle \nu \rangle$ and sum all the contributions. Fix $\sigma \in \langle \nu \rangle$. Let $(x_1, \ldots , x_m) \gamma^r \in N_G(M^m) \cap \Pi_{0,r,\sigma}$, then $M^m$ equals
\begin{align*}
    (M \times \ldots \times M)^{(x_1, \ldots , x_m) \gamma^r} = M^{x_{m-r+1}\tau} \times \ldots \times M^{x_m \tau} \times M^{x_1} \times \ldots \times M^{x_{m-r}}  .
\end{align*}
So $x_{m-r+1} \tau, \ldots , x_m \tau \in N_{S_n}(M) \cap (S_n - A_n)$ and $x_1, \ldots , x_{m-r} \in N_{S_n}(M) \cap A_n$. Since $N_{S_n}(M)$ is not contained in $A_n$, the sets $N_{S_n}(M) \cap A_n$ and $N_{S_n}(M) \cap (S_n-A_n)$ have the same cardinality.  Since $(x_1, \ldots , x_m) \gamma^r \in \Pi_{0,r}$, the $x_i$'s must also satisfy the equations of the definition of $\Pi_{0, r , \sigma}$. So for each equation $$ x_i x_{i+r} \ldots x_{i+m-r} \tau \in D_{\sigma(i)} ,$$ where $i = 1, \ldots , r$, we can freely choose the elements $x_{i+r}, \ldots , x_{m-r+i}$, with $| N_{S_n}(M) \cap A_n | = \frac{1}{2} |N_{S_n}(M)|$ choices for each, and only the elements $x_i$, $i=1,\ldots,r$, need to be chosen in order to satisfy the equation defining $\Pi_{0, r , \sigma}$, which is $x_iz_i \in D_{\sigma(i)}$, where $z_i = x_{i+r} \ldots x_{i+m-r} \tau \in N_{S_n}(M)$. Since
$$D_{\sigma(i)}z_i^{-1} \cap N_{S_n}(M) = (D_{\sigma(i)} \cap N_{s_n}(M))z_i^{-1},$$
there are $|D_{\sigma(i)} \cap N_{S_n}(M)|$ choices for each $x_i$, $i=1,\ldots,r$, and the result follows.

Assume now that $m$ is odd. Let $(x_1, \ldots , x_m) \gamma^2 \in N_G(M^m) \cap \Pi_{0,2}$. Then $M^m$ equals
\begin{align*}
 (M \times \ldots \times M)^{(x_1, \ldots , x_m) \gamma^2} 
 = M^{x_{m-1}\tau} \times M^{x_m \tau} \times M^{x_1} \times \ldots \times M^{x_{m-2}}.
\end{align*}
So $x_{m-1} \tau, x_m \tau \in N_{S_n}(M)\cap(S_n - A_n)$ and $x_1, \ldots , x_{m-2} \in N_{S_n}(M) \cap A_n$. Since $N_{S_n}(M)$ is not contained in $A_n$, we have $\frac{1}{2} |N_{S_n}(M)|$ choices for each of $x_{m-1}$ and $x_m$.  Now we can choose $x_2, \ldots , x_{m-2}$ freely in $N_{S_n}(M) \cap A_n$ and we need to choose $x_1$ in order to satisfy the equation that defines $\Pi_{0,2}$, which is $x_1 t \in C$, where $t = x_3 x_5 \ldots x_m \tau x_2 x_4 \ldots x_{m-1} \tau \in N_{S_n}(M)$. We can choose $x_1$ freely in $Ct^{-1} \cap N_{S_n}(M) = (C \cap N_{S_n}(M))t^{-1}$, so we have $|C \cap N_{S_n}(M)|$ choices for $x_1$. The result follows.
\end{proof}

\begin{cor} \label{pi_0r}
Assume that either $m$ is even or $r \neq 2$. Then $N_G(M^m) \cap \Pi_{0,r} = \varnothing$ if and only if $N_{S_n}(M) \cap D_i = \varnothing$ for at least one $i \in \{1,\ldots,r\}$. Moreover, if $m$ is odd and $r = 2$, then $N_G(M^m) \cap \Pi_{0,2} = \varnothing$ if and only if $N_{S_n}(M) \cap C = \varnothing$.
\end{cor}

\begin{lemma} \label{pi_i}
If $i \in I$, then $|N_G(M^m) \cap \Pi_i| = \left( \frac{1}{2} \; |N_{S_n}(M)| \right)^{m-1} \cdot |B_i \cap N_{S_n}(M)|$.
\end{lemma}

\begin{proof}
Let $(x_1, \ldots , x_m ) \gamma \in N_G(M^m) \cap \Pi_i$, then $M^m$ equals
$$(M^m)^{(x_1, \ldots , x_m)\gamma} = (M^{x_1} \times \ldots \times M^{x_m})^{\gamma} = M^{x_m \tau} \times M^{x_1} \times \ldots \times M^{x_{m-1}} .$$
So $x_m \tau \in N_{S_n}(M)\cap(S_n - A_n)$ and $x_1, \ldots , x_{m-1} \in N_{S_n}(M) \cap A_n$. Since $N_{S_n}(M)$ is not contained in $A_n$, the number of choices for $x_m$ is $\frac{1}{2} |N_{S_n}(M)|$. Now we can choose $x_2, \ldots , x_{m-1}$ freely in $N_{S_n}(M) \cap A_n$ and we need to choose $x_1$ in order to satisfy the equation that defines $\Pi_i$, which is $x_1 t \in B_i$, where $t = x_2 \ldots x_m \tau \in N_{S_n}(M)$. In other words, we can choose $x_1$ freely in $B_i t^{-1} \cap N_{S_n}(M) = (B_i \cap N_{S_n}(M))t^{-1}$ so the number of choices for $x_1$ is $|B_i \cap N_{S_n}(M)|$. The result follows.
\end{proof}

\begin{cor} 
If $i \in I$, then $N_G(M^m) \cap \Pi_i = \varnothing$ if and only if $B_i \cap N_{S_n}(M) = \varnothing$.
\end{cor}

The following proposition implies that condition (2) of Lemma \ref{EricLemma} holds.

\begin{prop} \label{unico}
Let $\pi \in \Pi$, then there is a unique $L \in \mathcal{C}$ such that $\pi \in L$.
\end{prop}

\begin{proof}
If $\pi \in \Pi_{0,r}$ for some prime $r$ that divides $2m$, then $M_{0,r}$ is the only subgroup in $\mathcal{M}$ that contains $\pi$. This follows from Corollary \ref{pi_0r} and the fact that no subgroup in $\mathcal{F}$ has non-empty intersection with all sets $D_1,\ldots,D_r$, because $n$-cycles do not belong to intransitive subgroups and $(n-2)$-cycles do not stabilize partitions with $2$ blocks.

Now suppose that $\pi \in \Pi_i$ for some $i \in I$. Then $\pi = (x_1, \ldots , x_m)\gamma$ with $x_1 \ldots x_m \tau \in B_i$, in particular $\pi \not \in M_{0,r}$ for every prime $r$ that divides $2m$. There is a unique $H \in \mathcal{F}$ such that $x_1 \cdots x_m \tau \in H$ and $H = N_{S_n}(M)$ where $M=H \cap A_n \in \mathcal{E}$. Suppose that $\pi = (x_1, \ldots , x_m)\gamma$ belongs to $N_G(M^{a_1} \times \ldots \times M^{a_m})$. Then $M^{a_1} \times \ldots \times M^{a_m}$ equals
\begin{align*}
    (M^{a_1} \times \ldots \times M^{a_m})^{(x_1, \ldots , x_m)\gamma} = M^{a_m x_m \tau} \times M^{a_1 x_1} \times \ldots \times M^{a_{m-1} x_{m-1}} .
\end{align*}
So, for $i$ with $1 \leqslant i \leqslant m-1$, $a_i x_i a_{i+1}^{-1} \in H$ and $a_m x_m \tau a_1^{-1} \in H$.
Multiplying all these elements starting from the $i$-th one, we have
$$a_i x_i x_{i+1} \ldots x_m \tau x_1 x_2 \ldots x_{i-1} a_i^{-1} \in H,$$
which can be written as
$$(x_1 \ldots x_m \tau)^{x_1 \ldots x_{i-1}} = x_i x_{i+1} \ldots x_m \tau x_1 x_2 \ldots x_{i-1} \in H^{a_i}.$$
In particular $x_1 \ldots x_m \tau \in H^{a_1}$. Since $\mathcal{F}$ is closed under conjugation, the uniqueness of $H$ implies that $H^{a_1} = H$, therefore $a_1 \in N_{S_n}(H) = H$, being $H$ a maximal subgroup of $S_n$. Now we can rewrite the above equations as
$$a_i \in H x_1 x_2 \ldots x_{i-1}, \hspace{.5 cm} \forall i = 2, \ldots, m. $$
It follows that $M^{a_1}=M$ and $M^{a_i}=M^{x_1 \ldots x_{i-1}}$ for all $i=2,\ldots,m$, hence the only subgroup in $\mathcal{C}$ that contains $\pi$ is
$$N_G(M \times M^{x_1} \times M^{x_1x_2} \times \ldots \times M^{x_1 \ldots x_{m-1}}).$$
This concludes the proof.
\end{proof}

In order to conclude the proof of Theorem \ref{main_1}, we are left to show that condition (3) of Lemma \ref{EricLemma} holds, in other words that $d(H) < 1$ for every maximal subgroup $H$ of $G$ not in $\mathcal{C}$.

\begin{lemma} \label{pi_2}
For $n \geqslant 30$, the value of $|\Pi_{0,2}|$ is smallest when $m$ is even. Moreover, if $g \in G$, then $|N_G(M^m)^g \cap \Pi_{0,2} |/|\Pi_{0,2}| \leqslant 2n(n-2)/|S_n:N_{S_n}(M)|^m$.
\end{lemma}
\begin{proof}
We have $|D_1|=n!/(2(n-2))$, $|D_2|=(n-1)!$ and $|C| = n!/(p(n-p))$.
$$|\Pi_{0,2}| = \left\{
\begin{array}{ll} 
2 |A_n|^{m-2} |D_1| |D_2| = |A_n|^{m} \cdot 4/(n(n-2)) & \mbox{if } m \mbox{ is even} \\
|A_n|^{m-1} |C| = |A_n|^{m} \cdot 4/(2p(n-p)) & \mbox{if } m \mbox{ is odd}.
\end{array} \right.$$
Note that $2p(n-p) < 2(2n/3)^2 \leqslant n(n-2)$ being $n/3 < p < 2n/3$ and $n \geqslant 30$. So the value of $|\Pi_{0,2} |$ is smallest when $m$ is even. We now prove the stated inequality. Since $\Pi_{0,2}$ is closed under conjugation, we may assume that $g=1$. By Lemma \ref{pi_0},
$$\frac{|N_G(M^m) \cap \Pi_{0,2} |}{|\Pi_{0,2}|} = \left\{
\begin{array}{ll} 
 \left( \frac{|N_{S_n}(M)|}{|S_n|} \right)^{m-2} \cdot \frac{|D_1 \cap N_{S_n}(M)| |D_2 \cap N_{S_n}(M)|}{|D_1||D_2|} & \mbox{if } m \mbox{ is even} \\
\left( \frac{| N_{S_n}(M)|}{|S_n|} \right)^{m-1} \cdot \frac{|C  \cap N_{S_n}(M)|}{|C|} & \mbox{if } m \mbox{ is odd}.
\end{array} \right.$$
The inequality in the statement follows by using the fact that the size of the intersection of any one of $D_1$, $D_2$, $C$ with $N_{S_n}(M)$ is at most $|N_{S_n}(M)|$.
\end{proof}

\begin{lemma} \label{a!b}
If $d$ is a divisor of $n$ such that $2 \leqslant d \leqslant n/2$ then $(n/d)!^d \cdot d! \leqslant 2(n/2)!^2$.
\end{lemma}
\begin{proof}
We do as in the proof of \cite[Lemma 2.1]{Attila}. Assume first that $d \leqslant n/d$.
\begin{align*}
    (n/d)!^d \cdot d! & 
    \leqslant (n/d)!^{2} \cdot 2 \cdot ((n/d)! \cdot d)^{d - 2}
    \leqslant (n/d)!^{2} \cdot 2 \cdot ((n/d)! \cdot n/d)^{d - 2} \\
    & \leqslant (n/d)!^{2} \cdot 2 \cdot ((n/d)^{n/d})^{d - 2} 
    = (n/d)!^{2} \cdot 2 \cdot (n/d)^{2(n/2 - n/d)} \\
    & \leqslant (n/d)!^{2} \cdot ((n/d) +1)^{2} \cdot \ldots \cdot (n/2)^{2} \cdot 2 = 2 (n/2)!^{2},
\end{align*}
where, in the fourth inequality, we used that $r! \leqslant r^{r-1}$, for all $r \geqslant 2$.

Suppose now that $d > n/d$. Since $2 \leqslant n/d \leqslant n/2$, exchanging the role of $n/d$ and $d$ in the above inequality we obtain $d!^{n/d} \cdot (n/d)! \leqslant 2 (n/2)!^{2}$. If $a > b \geqslant 2$ are integers, then $a!^b \cdot b! > b!^a \cdot a!$, since
\begin{align*}
   a!^{b-1} & = \left( a \cdot (a-1) \cdot \ldots \cdot (b+1) \right)^{b-1} \cdot b!^{b-1} \geqslant \left( (b+1)^{(a-b)} \right)^{b-1} \cdot b!^{b-1} \\ & > b^{(a-b)(b-1)} \cdot b!^{b-1} \geqslant b!^{(a-b)} \cdot b!^{b-1} = b!^{a-1}.
\end{align*} 
Applying this to $a=d$, $b=n/d$ we have $(n/d)!^d \cdot d! < d!^{n/d} \cdot (n/d)! \leqslant 2 (n/2)!^{2}$. 
\end{proof}

\begin{lemma} \label{ordemnormalizador}
Let $H$ be a maximal subgroup of $S_n$ such that $H \notin \mathcal{F}$ and fix $i \in I$, $M \in \mathcal{E}_i$. Then either $|H| \leqslant |N_{S_n}(M)|$ or $H \cap B_i = \varnothing$.
\end{lemma}
\begin{proof}
By the O'Nan–Scott Theorem, the maximal subgroups of $S_n$ are of one of the following types: (1) primitive, (2) maximal intransitive, isomorphic to $S_k \times S_{n-k}$ for some $k \in \{1,\ldots,n/2-1\}$ and (3) maximal imprimitive, isomorphic to $S_a \wr S_b$ for $2 \leqslant a,b < n$ with $ab=n$.
If $H$ is intransitive then $H \cong S_k \times S_{n-k}$ with $n/3 \leqslant k < n/2$, therefore $|H| \leqslant (n/3)! (2n/3)! \leqslant (n/3 -1)! (2n/3 +1)! \leqslant |N_{S_n}(M)|$ if $N_{S_n}(M)$ is intransitive. On the other hand, if $N_{S_n}(M)$ is transitive, then $i=-1$ and $H \cap B_{-1} = \varnothing$.

Now suppose that $H$ is transitive. If $H$ is imprimitive then $|H| = (n/d)!^d \cdot d!$, where $d$ is a divisor of $n$, $d \neq 1, 2, n$. By Lemma \ref{a!b}, if $N_{S_n}(M)$ is imprimitive, then $|H| = (n/d)!^d \cdot d! \leqslant (n/2)!^{2} \cdot 2! = |N_{S_n}(M)|$. 
If $N_{S_n}(M)$ is intransitive, then $|H| = (n/d)!^d \cdot d! \leqslant (n/2)!^{2} \cdot 2! \leqslant (n/3 -1)! (2n/3 +1)! \leqslant |N_{S_n}(M)|$. If $H$ is primitive then either $H=A_n$, in which case $H \cap B_i = \varnothing$, or $H \neq A_n$, in which case $|H| < 4^n$ by \cite{primitivo}. Since $n \geqslant 30$ we have $4^n \leqslant (n/2)!^{2} \cdot 2 \leqslant (n/3 -1)! (2n/3 +1)!$ and the result follows.
\end{proof}

\begin{prop} \label{d < 1}
Let $H$ be a maximal subgroup of $G$ not in $\mathcal{C}$. Then $d(H) < 1$. 
\end{prop}

\begin{proof}
Assume $H$ has product type. Then $H$ is conjugate to $N_G(M^m)$ where $M$ is the intersection between $A_n$ and a maximal subgroup of $S_n$ not of the form $A_n$ nor $S_{n/2} \wr S_2$ nor $S_i \times S_{n-i}$, $i = 1, 2, \ldots, n/3 - 1$, so that $|N_{S_n}(M)| \leqslant (n/3)! \; (2n/3)!$ by Lemma \ref{a!b} and the fact that $2 (n/2)!^2 \leqslant (n/3)! \; (2n/3)!$ being $n \geqslant 30$. If $M$ is primitive, by \cite{primitivo} we have $|N_{S_n}(M)| < 4^n \leqslant 2 (n/2)!^2 \leqslant (n/3)! \; (2n/3)!$. Since $\Pi_j$ is closed under conjugation for all $j \in J$, we have $|H \cap \Pi_j| = |N_G(M^m) \cap \Pi_j|$ for all $j \in J$. We will use Stirling's inequalities, which are valid for all $k \geqslant 2$:
 $$\sqrt{2 \pi k} \; (k/e)^k \leqslant k! \leqslant  e \sqrt{k} \; (k/e)^k. $$
Assume that either $m$ is even or $r \neq 2$. By Lemma \ref{pi_0} and the fact that $\Pi_{0,r} \subseteq M_{0,r}$,
\begin{align*}
\displaystyle \frac{|H \cap \Pi_{0,r} |}{|M_{0,r} \cap \Pi_{0,r} |} 
& = \frac{ r \cdot \left( \frac{1}{2} \; | N_{S_n}(M)| \right)^{m-r}  \cdot \prod_{i = 1}^r |D_i \cap N_{S_n}(M) |}{r \cdot |A_n|^{m-r} \cdot \prod_{i = 1}^r |D_i|} \\
 & \leqslant  \left( \frac{|N_{S_n}(M)|}{|S_n|} \right)^{m-r} \cdot \prod_{i =1}^r \frac{|N_{S_n}(M) |}{|D_i|}
  =  \left( \frac{|N_{S_n}(M)|}{|S_n|} \right)^{m} \cdot 2(n-2)n^{r-1} \\
 & \leqslant  \left( \frac{(n/3)! \; (2n/3)!}{n!} \right)^{m} \cdot 2n^{r} 
  \leqslant 2 \cdot \left( \frac{2^{2/3}}{3} \right)^{nm} \cdot \left( \frac{n e^2 \sqrt{n}}{3 \sqrt{\pi}} \right)^{m}.
\end{align*}

By Lemma \ref{pi_2}, $|N_G(M^m) \cap \Pi_{0,2} |/|\Pi_{0,2}| \leqslant 2n(n-2)/|S_n:N_{S_n}(M)|^m$, then we have the above inequality also in the case $r = 2$ when $m$ is odd.

According to the proof of \cite[Lemmas 5.4, 5.5, 5.9, 5.10]{Eric}, the largest value of $\sum_{i \; \in \; I} \frac{|B_i \cap N_{S_n}(M)|}{|B_i \cap N_{S_n}(M_i)|}$ is obtained by substituting $n=30$ in the expression $\frac{3n^2+27n+54}{4n^2-9}$, so it is less than $0.9925$. By Lemma \ref{ordemnormalizador},
\begin{align*}
\displaystyle \sum_{i \; \in \; I} \frac{|H \cap \Pi_i |}{|N_G(M_i^m) \cap \Pi_i |} &= \displaystyle \sum_{i \; \in \; I} \frac{ \left( \frac{1}{2} \; |N_{S_n}(M)| \right)^{m-1} \cdot |B_i \cap N_{S_n}(M)|}{\left( \frac{1}{2} \; |N_{S_n}(M_i)| \right)^{m-1} \cdot |B_i \cap N_{S_n}(M_i)|} \\
& \leqslant \sum_{i \; \in \; I} \frac{|B_i \cap N_{S_n}(M)|}{|B_i \cap N_{S_n}(M_i)|} < 0.9925
\end{align*}
We obtain
\begin{align*}
d(H) & = \displaystyle \sum_{r \in P(2m)} \frac{|H \cap \Pi_{0,r} |}{|M_{0,r} \cap \Pi_{0,r} |} +  \sum_{i \; \in \; I} \frac{|H \cap \Pi_i |}{|N_G(M_i^m) \cap \Pi_i |} \\
& < 2m \cdot \left[ \left( \frac{2^{2/3}}{3} \right)^n \cdot \frac{ n e^2 \sqrt{n}}{3 \sqrt{\pi} } \right]^{m} + 0.9925
\end{align*}

This is less than $1$ being $n \geqslant 30$.

We now turn our attention to the maximal subgroups of $G$ of diagonal type and supplementing the socle $N$. Let $H$ be such a subgroup. Recall that $H \cap N = \Delta_{\varphi}$ has order $|A_n|^{m/t}$ where $t$ is a prime divisor of $m$. We have  
\begin{align*}
d(H) & = \displaystyle \sum_{r \in P(2m)} \frac{|H \cap \Pi_{0,r} |}{|M_{0,r} \cap \Pi_{0,r} |} +  \sum_{i \; \in \; I} \frac{|H \cap \Pi_i |}{|N_G(M_i^m) \cap \Pi_i |} \\ & \leqslant \displaystyle |H| \cdot \left( \sum_{r \in P(2m)} \frac{1}{|\Pi_{0,r} |} +  \sum_{i \; \in \; I} \frac{1}{|N_G(M_i^m) \cap \Pi_i |}\right).
\end{align*}
Since $HN=G$, we have $C_{2m} \cong G/N = HN/N \cong H/H \cap N$, hence 
$$|H| = |H:H \cap N| \cdot |H \cap N| = 2m \cdot |\Delta_\varphi| = 2m \cdot \left( n!/2 \right)^{m/t}.$$

Assume first that either $m$ is even or $r \neq 2$. Since $2 \leqslant r \leqslant m$,
\begin{align*}
|\Pi_{0,r}| &=  r \cdot |A_n|^{m-r} \cdot \prod_{i=1}^r |D_i| = r \cdot \left( \frac{n!}{2} \right)^{m-r} \cdot  \frac{n!}{n} \cdot \left( \frac{n!}{2(n-2)} \right)^{r-1} \\ 
& = \frac{r \cdot n!^m}{2^{m-1} n (n-2)^{r-1}} \geqslant \frac{2 \cdot n!^m}{2^{m-1} n^r} \geqslant \frac{n!^m}{2^{m-2} n^m}.
\end{align*}
By Lemma \ref{pi_2}, the smallest value of $|\Pi_{0,2} |$ is when $m$ is even, so the above inequality for $|\Pi_{0,r}|$ holds in all cases.

Fix $M_i \in \mathcal{E}_i$ for all $i \in I$. The smallest possible order of $N_{S_n}(M_i)$, $i \in I$, is when $M_i$ is imprimitive with two blocks, so $|N_{S_n}(M_i)| \geqslant 2 \left(n/2 \right)!^2$. Since $B_i \cap N_{S_n}(M_i) \neq \varnothing$, by Lemma \ref{pi_i} we have
\begin{align*}
    |N_G(M_i^m) \cap \Pi_i| & = \left( \frac{1}{2} \; |N_{S_n}(M_i)| \right)^{m-1} |B_i \cap N_{S_n}(M_i)|
    \geqslant \left(n/2\right)!^{2(m-1)}.
\end{align*}
We deduce that
\begin{align*}
    d(H) & \leqslant \displaystyle 2m \left( \frac{n!}{2} \right)^{m/t} \cdot \left( \sum_{r \in P(2m)} \frac{2^{m-2} n^m}{(n!)^m} +  \sum_{i \; \in \; I} \frac{1}{\left(n/2 \right)!^{2(m-1)}}\right) \\
    & \leqslant 2m \left( \frac{1}{2} (n/e)^n e \sqrt{n} \right)^{m/t} \left( 2^{m-1} \frac{m n^m}{(n/e)^{nm}}+\frac{n}{(n/(2e))^{n(m-1)}}\right) < 1
\end{align*}
for $m \geqslant 3$ and $n \geqslant 30$, where we used the fact that $t \geqslant 2$ and $t=3$ if $m=3$.

Now assume that $m=2$. We will show that $d(H)=0$ by proving that $H \cap \Pi_{0,2}$ and $H \cap \Pi_i$ are empty for all $i \in I$. We have $H=N_G(\Delta_{\varphi})$ where $\Delta_{\varphi} = \{ (\alpha, \alpha^{\varphi}) : \alpha \in A_n \}$ for $\varphi \in \Aut(A_n) \cong S_n$ and
$$\begin{array}{l}
\Pi_{0,2} = \{ (x_1, x_2) \gamma^2\ :\ x_1 \tau \in D_1,\ x_2 \tau \in D_2 \} \cup \{ (x_1, x_2) \gamma^2\ :\ x_1 \tau \in D_2,\ x_2 \tau \in D_1 \}, \\
\Pi_i = \{ (x_1,x_2) \gamma : x_1x_2 \tau \in B_i \}, \hspace{.5cm} i \in I.
\end{array}$$
 
 For $i \in I$ we have that if $(x_1,x_2)\gamma \in H \cap \Pi_i$ then $$(\alpha, \alpha^{\varphi})^{(x_1,x_2)\gamma} = (\alpha, \alpha^{\varphi})^{(x_1,x_2)(1, \tau)\delta} = (\alpha^{x_1}, \alpha^{\varphi x_2 \tau})^{\delta} = (\alpha^{\varphi x_2 \tau}, \alpha^{x_1}) \in \Delta_{\varphi},$$
 
So $\varphi x_2 \tau \varphi = x_1$, equivalently $(\varphi x_2 \tau)^2 = x_1x_2 \tau$ which is false since $(\varphi x_2 \tau)^2 \in A_n$ and $x_1x_2\tau \notin A_n$. Therefore $H \cap \Pi_i = \varnothing$ for all $i \in I$. 
 
 If $(x_1,x_2)\gamma^2 \in H \cap \Pi_{0,2}$ then, for all $\alpha \in A_n$, 
 $$ ( \alpha, \alpha^{\varphi})^{(x_1,x_2)\gamma^2} = (\alpha, \alpha^{\varphi})^{(x_1,x_2)(\tau, \tau)} =  (\alpha^{x_1 \tau}, \alpha^{\varphi x_2 \tau}) \in \Delta_{\varphi}. $$
So $x_1\tau \varphi = \varphi x_2 \tau$, i.e. $\varphi^{-1} x_1 \tau \varphi = x_2 \tau$. This is a contradiction because $x_1 \tau$ and $x_2 \tau$ are not conjugated in $S_n$ by definition of $\Pi_{0,2}$. Therefore $H \cap \Pi_{0,2} = \varnothing$.
\end{proof}

\section{Proof of Theorem 2} \label{section_main_2}

For the calculation of $\omega(G)$, we follow the same strategy used in \cite{MFA}. We use the following result that can be found in \cite{LocalLemma1}. The formulation we use is taken from \cite[Corollary 5.1.2]{LocalLemma} (the ``symmetric case''). Given an event $E$ of a probability space, we denote by $P(E)$ its probability and by $\overline{E}$ its complement. As usual $e$ denotes the base of the natural logarithm.

\begin{teor} [Lovász Local Lemma]
Let $E_1, E_2, \ldots, E_n$ be events in an arbitrary probability space. Let $(V, E)$ be a directed graph, where $V = \{ 1, \ldots, n \}$ is the set of vertices, and assume that,
for every $i \in V$ , the event $E_i$ is mutually independent of the set of events $E_j$ such that $(i, j) \notin E$. Let $d$ be the maximum valency of a vertex of the graph $(V, E)$. If for every $i \in V$ $$P(E_i) \leqslant \frac{1}{e(d+1)} $$ then $P \left( \bigcap_{i \in V} \overline{E_i} \right) > 0$.
\end{teor}

The mutual independence condition mentioned in the Lovász Local Lemma means the following:
$$ P \left( E_i | \bigcap_{j \in S} \overline{E_j} \right) = P(E_i), $$
for all $i \in V$ and for all subset $S$ of $\{ j \in V : (i, j) \notin E \}$.

Define $$\mathcal{N} = \{ N_G(M \times M^{a_2} \times \ldots \times M^{a_m}) \ :\ M \in \mathscr{F} \}, $$ 
where $\mathscr{F}$ is the family of maximal imprimitive subgroups of $A_n$ with $2$ blocks, $(S_{n/2} \wr S_2) \cap A_n$, and $a_2, \ldots, a_m \in A_n$. Note that if $H \in \mathcal{N}$ then $H$ is conjugate to $N_G(M^m)$ in $G$, for some $M \in \mathscr{F}$. The subgroups of $G$ contained in $\mathcal{N}$ are maximal in $G$ by \cite[Proposition 1.1.44]{Spagnoli} and  \cite{LPS}.

Let $B$ be the set of $n$-cycles in $S_n$ and let $\Pi$ be the set of elements of $G$ of the form $(x_1,\ldots,x_m) \gamma$ with the property that $x_1 \ldots x_m \tau \in B$. Note that these sets are precisely what are called $B_{-1}$ and $\Pi_{-1}$ in Section \ref{section_main_1}. 

\begin{lemma} \label{piclass}
    $\Pi$ is a conjugacy class of $G$.
\end{lemma}
\begin{proof}
Note that $B \nsubseteq A_n$, so there is $z \in A_n$ such that $z \tau \in B$. It follows that $\pi := (z, 1, \ldots, 1)\gamma \in \Pi$. We prove that $\Pi$ is the conjugacy class of $\pi$ in $G$. Let $(x_1, \ldots, x_m)\gamma \in \Pi$, we will prove that this element is conjugate to $\pi$ in $G$. There exists $a \in S_n$ with $(x_1 \ldots x_m \tau)^a = z \tau$. If $a \not \in A_n$, then $b= x_1 \ldots x_m \tau a \in A_n$ and $(x_1 \ldots x_m \tau)^b = (x_1 \ldots x_m \tau)^a$, so we may assume that $a \in A_n$. Set $y_1 := a$ and $y_i := x_i \ldots x_m \tau a \tau$ for $i = 2,\ldots,m$. Then $((x_1, \ldots , x_m) \gamma)^{(y_1, \ldots , y_m)}$ equals
\begin{align*}
(y_1^{-1} x_1 y_2, y_2^{-1} x_2 y_3 , \ldots, y_{m-1}^{-1} x_{m-1} y_m , y_m^{-1} x_m \tau y_1 \tau) \gamma = (z,1,\ldots,1) \gamma = \pi.
\end{align*}
This concludes the proof.
\end{proof}

For $H \in \mathcal{N}$ and $K \leqslant G$, define 
$$C(H) = \Pi \cap H, \hspace{1cm} f_H(K) = \frac{ | C(H) \cap K |}{ |C(H)| }.$$
Let $g \in G$ be such that $H = (N_G(M^m))^g$. By Lemmas \ref{pi_i} and \ref{piclass}, 
\begin{align*}
| C(H) | & =  |H \cap \Pi | = | (N_G(M^m))^g \cap \Pi | = | N_G(M^m) \cap \Pi | \\
& = \left( \frac{1}{2} \; |N_{S_n}(M)| \right)^{m-1} \cdot |B \cap N_{S_n}(M)| = 2/n \cdot \left( n/2 \right)!^{2m}.
\end{align*}
Since $H$ is a non-normal maximal subgroup of $G$, it is self-normalizing. Since $\mathcal{N}$ is the conjugacy class of $H$ in $G$,
\begin{align*}
    l=|\mathcal{N}| = |G:H| = \frac{(n!/2)^m \cdot 2m}{(n/2)!^{2m} \cdot 2m} = \frac{1}{2^m} \binom{n}{n/2}^m < 2^{m(n-1)}.
\end{align*}
Define the graph $\Gamma$ whose vertices are the two-element subsets $v = \{H_1, H_2 \}$ of $\mathcal{N}$, with $H_1 \neq H_2$. There is an edge between two vertices $v$ and $w$ if $v \cap w \neq \varnothing$. Every vertex of $\Gamma$ has valency $d = 2(l-2) < 2^{m(n-1)+1}$. Choose $g_H \in C(H)$ uniformly and independently, for all $H \in \mathcal{N}$, and let $E_v$ be the event $\langle g_{H_1}, g_{H_2} \rangle \neq G$, equivalently $\langle g_{H_1}, g_{H_2} \rangle$ is contained in a maximal subgroup of $G$. It is easy to see that the mutual independence condition is satisfied (see also \cite[Section 3]{MFA}). Our aim is to prove that $P(E_v) \leqslant 1/(e(d+1))$ for every vertex $v$ of $\Gamma$. If this is true, then the Local Lemma implies that there exists a choice of $g_H$ in each $C(H)$, $H \in \mathcal{N}$, with the property that $\langle g_{H_1},g_{H_2} \rangle = G$ for all $H_1 \neq H_2$ in $\mathcal{N}$, therefore these elements form a clique of the generating graph of $G$, in other words $\omega(G) \geqslant |\mathcal{N}|$. This, together with \cite[Theorem 1 (3)]{Martino}, gives the claim of Theorem \ref{main_2}.

In the following discussion we will talk about the various types of maximal subgroups of $G$, which we described in Section \ref{section_main_1}.

Let $\mathcal{M}_1$ be the family of maximal intransitive subgroups of $S_n$, $\mathcal{M}_2$ the family of primitive maximal subgroups of $S_n$ different from $A_n$, $\mathcal{M}_j$ the family of maximal imprimitive subgroups of $S_n$ with $j$ blocks for $j \in \{3,4\}$, $\mathcal{M}_5$ the family of maximal imprimitive subgroups of $S_n$ with at least $5$ blocks. Let $\mathcal{H}$ be the family of all maximal subgroups of $G$ not in $\mathcal{N}$ and $J = \{1, 2, 3, 4, 5, 6 \}$. We write $\mathcal{H}$ as the union $\mathcal{H}_1 \cup \ldots \cup \mathcal{H}_6$ where the $\mathcal{H}_j$'s are defined as follows. For $j$ with $1 \leqslant j \leqslant 5$, $\mathcal{H}_j$ is the subset of $\mathcal{H}$ consisting of subgroups of the form $N_G(M \times M^{a_2} \times \ldots \times M^{a_m})$, where $a_2,\ldots,a_m \in A_n$, $N_{S_n}(M) \in \mathcal{M}_j$ and $N_{S_n}(M) \cap A_n = M$. $\mathcal{H}_6$ is the family of maximal subgroups of $G$ of diagonal type.

Fix a vertex $v = \{ H_1, H_2 \}$ of $\Gamma$. For $j \in J$, let $E_{v,j}$ be the probability that $\langle g_{H_1}, g_{H_2} \rangle$ is contained in a member of $\mathcal{H}_j$. We clearly have
$$ P(E_v) \leqslant \sum_{j \in J} P(E_{v, j}).$$
Let $[H]$ be the conjugacy class in $G$ of a subgroup $H$ of $G$ and $m_{H_i}([H])$ the number of different conjugates of $H$ that contain a fixed element of $C(H_i)$, $i = 1,2$. This is well defined by Lemma \ref{piclass}. In the following sum, $[H]$ varies in the set of conjugacy classes of elements of $\mathcal{H}_j$.
Arguing as in \cite{MFA} we have, for $j \in J$,
\begin{align*}
    P(E_{v,j}) & \leqslant \sum_{[H]} m_{H_1}([H]) \max_{K \in [H]} (f_{H_2}(K)).
\end{align*}
Let $c_{v,j}$ the number of conjugacy classes of subgroups in $\mathcal{H}_j$ such that there exists $H$ in such a class such that $H \cap C(H_1) \neq \varnothing$ and $H \cap C(H_2) \neq \varnothing$. We deduce that
\begin{align} \label{ineq_cvj} \tag{$\star$}
    P(E_{v, j}) \leqslant c_{v,j} \cdot \min_{\{ i_1, i_2 \} = \{ 1, 2 \}} \left( \max_{H \in \mathcal{H}_j, K \in [H]} (m_{H_{i_1}}([H]) \cdot f_{H_{i_2}}(K) ) \right) .
\end{align}
Let $s_{v,j}$ be the number of subgroups $H$ in $\mathcal{H}_j$ such that $H \cap C(H_1) \neq \varnothing$ and $H \cap C(H_2) \neq \varnothing$. Then
\begin{align} \label{ineq_svj} \tag{$\star \star$}
    P(E_{v,j}) \leqslant \sum_{H \in \mathcal{H}_j} f_{H_1}(H) f_{H_2}(H) \leqslant s_{v,j} \cdot \max_{H \in \mathcal{H}_j}(f_{H_1}(H) \cdot f_{H_2}(H)) .
\end{align}
\begin{lemma} \label{cvj}
Let $v = \{ H_1, H_2 \}$ be a vertex of $\Gamma$. Then $c_{v,2} \leqslant n$ for large enough $n$, $c_{v,j} \leqslant 1$ for $j \in \{ 3, 4 \}$, $c_{v,5} \leqslant 2 \sqrt{n}$ and $c_{v,6} \leqslant m \cdot 2^m$.
\end{lemma}
The bound $c_{v,2} \leqslant n$ depends on the classification of finite simple groups.
\begin{proof}
Note that $c_{v,j}$ is less than or equal to the number of conjugacy classes of subgroups in $\mathcal{H}_j$. Also, if $H \in \mathcal{H}$ then we can write $H=N_G(H \cap N)$ and this allows to reduce to counting $G$-conjugacy classes of subgroups of the form $H \cap N$ in $N$. 
Also note that if $M$ and $L$ are conjugate in $A_n$, then $N_G(M^m)$ and $N_G(L^m)$ are conjugate in $G$ by an element of the form $(c,c,\ldots,c) \in A_n^m$ such that $M^c=L$. Therefore, for $j$ with $1 \leqslant j \leqslant 5$, the number of conjugacy classes of subgroups in $\mathcal{H}_j$ is less than or equal to the number of conjugacy classes of subgroups of $S_n$ belonging to $\mathcal{M}_j$. Therefore, for $j \neq 6$, we can use the bounds for $c_{v,j}$ calculated in \cite[Lemma 5]{MFA}. In other words $c_{v,2} \leqslant n$ for large enough $n$, $c_{v,j} \leqslant 1$ for $j \in \{ 3, 4 \}$ and $c_{v,5} \leqslant 2 \sqrt{n}$.

It remains to bound $c_{v,6}$. We will use the fact that if $X \leqslant Y$ are finite groups with $Y$ acting on a finite set $\Omega$, then denoting by $u_X$ the number of $X$-orbits and by $u_Y$ the number of $Y$-orbits of this action, we have $u_Y  \leqslant u_X \leqslant |Y:X| \cdot u_Y$. Since $n$ is larger than $6$, $\Aut(A_n) \cong S_n$, therefore any two isomorphic diagonal subgroups $\Delta_{\varphi_1}$, $\Delta_{\varphi_2}$ of the socle $N=A_n^m$ are conjugate in the group $S_n^m \rtimes \langle \delta \rangle$, which contains $G$, via an element of $S_n^m$. It follows that the number of $G$-classes of isomorphic diagonal subgroups is at most the number of $A_n^m$-classes, which is at most $|S_n:A_n|^m=2^m$. We know that the number of isomorphism classes of diagonal subgroups equals the number of prime divisors of $m$ (see Section \ref{section_main_1}). Therefore $c_{v,6} \leqslant m \cdot 2^m$.
\end{proof}

\begin{lemma} \label{sv4}
Let $v$ be a vertex of $\Gamma$ and assume that $4$ divides $n$. Then $s_{v,4} \leqslant 1$.
\end{lemma}
\begin{proof}

Let $v = \{ H_1, H_2 \}$ and let $H \in \mathcal{H}_4$. Write 
$$\begin{array}{lll}
H = N_G(R^{b_1} \times \ldots \times R^{b_m}) \in \mathcal{H}_4, & &
H_i = N_G(M_i^{a_{i1}} \times \ldots \times M_i^{a_{im}}) \in \mathcal{N},
\end{array}$$
for $i=1,2$, where each $a_{ij}$ and each $b_j$ belongs to $A_n$, $N_{S_n}(M_i)$ is a maximal imprimitive subgroup of $S_n$ with $2$ blocks for $i=1,2$ and $N_{S_n}(R)$ is a maximal imprimitive subgroup of $S_n$ with $4$ blocks. Suppose that $H \cap C(H_i) = H \cap \Pi \cap H_i \neq \varnothing$ for $i = 1, 2$. We need to show that $H$ is uniquely determined by these conditions, in other words, that each $R^{b_j}$ is uniquely determined. By \cite[Proof of Lemma 5]{MFA}, it is enough to prove that $B \cap N_{S_n}(M_i^{a_{ij}}) \cap N_{S_n}(R^{b_j}) \neq \varnothing$ for $i=1,2$ and for $j=1,\ldots,m$.

Fix $i \in \{1,2\}$ and let $h = (x_1, \ldots, x_m) \gamma \in H \cap C(H_i) = H \cap H_i \cap \Pi$. Since $h \in \Pi$, by definition $x_1 \ldots x_m \tau \in B$. On the other hand, being $h \in H$, $R^{b_1} \times \ldots \times R^{b_m}$ equals
\begin{align*}
    (R^{b_1} \times \ldots \times R^{b_m})^{(x_1, \ldots, x_m)\gamma}
    = R^{b_m x_m \tau} \times R^{b_1 x_1} \times R^{b_2 x_2} \times \ldots R^{b_{m-1} x_{m-1}} .
\end{align*}
We deduce that $b_m x_m \tau b_1^{-1} \in N_{S_n}(R)$ and $b_j x_j b_{j+1}^{-1} \in N_{S_n}(R)$ for $j=1,\ldots,m-1$. Fix $j \in \{1,\ldots,m\}$. Multiplying everything starting from the $j$-th term, we have 
$$b_j x_j x_{j+1} \ldots x_m \tau x_1 x_2 \ldots x_{j-1} b_j^{-1} \in N_{S_n}(R).$$
It follows that the element $x := x_j x_{j+1} \ldots x_m \tau x_1 x_2 \ldots x_{j-1}$ belongs to $N_{S_n}(R^{b_j})$. Since $h \in H_i$, the same argument shows that $x$ belongs to $N_{S_n}(M_i^{a_{ij}})$. Furthermore
$$x= (x_j x_{j+1} \cdots x_m \tau) \cdot x_1 \cdots x_m \tau \cdot (x_j x_{j+1} \cdots x_m \tau)^{-1},$$
so $x$ belongs to $B$. Therefore $x \in B \cap N_{S_n}(M_i^{a_{ij}}) \cap N_{S_n}(R^{b_j})$.
\end{proof}

\begin{lemma} \label{conjnm}
Let $L \leqslant G$ and $g \in \Pi$, then the number of conjugates of $L$ containing $g$ is at most $nm$.
\end{lemma}
\begin{proof}
We argue as in the proof of \cite[Lemma 4]{Blackburn}. Let $a(L)$ the number of conjugates of $L$ containing $g$. Note that $a(L)$ does not depend on $g$ because $\Pi$ is a conjugacy class in $G$. Consider the set $R$ of pairs $(h, H)$ such that $h \in H \cap \Pi$ and $H$ is conjugated to $L$ in $G$. On the one hand, since $\Pi$ is a conjugacy class of $G$, $|R| = |\Pi| \cdot a(L)$. On the other hand, since $L$ has $|G : N_G(L)|$ conjugates in $G$ and $|L^g \cap \Pi|=|L \cap \Pi|$ for all $g \in G$, $|R| = |G : N_G(L)| \cdot |L \cap \Pi| \leqslant |G : L| \cdot |L| = |G|$. Therefore $|\Pi| \cdot a(L) \leqslant |G|$ hence
$$a(L) \leqslant \frac{|G|}{|\Pi|} = \frac{2m \cdot (n!/2)^m}{(n-1)! \cdot (n!/2)^{m-1}} = n m .$$
This concludes the proof.
\end{proof}
Fix a vertex $v=\{H_1,H_2\}$ of $\Gamma$ and let $i \in \{1,2\}$, $H:=H_i$. By Lemma \ref{conjnm}, $m_{H}([K]) \leqslant n m$ for all $K \leqslant G$. We now bound $f_{H}(K)=|C(H) \cap K|/|C(H)|$ for $K \in \mathcal{H}_j$ and $P(E_{v,j})$ for $j = 1, \ldots, 6$. Since $\Pi$ is closed under conjugation, when bounding $f_H(K)$ we may assume that $H = N_G(L^m)$ where $L$ is a maximal imprimitive subgroup of $A_n$ with $2$ blocks. As in Section \ref{section_main_1}, we will use Stirling's inequalities. By Lemma \ref{pi_i}, $C(H) = H \cap \Pi$ has size $(2/n) \cdot (n/2)!^{2m} \geqslant (2/n) (n/(2e))^{nm}$.

\begin{enumerate}
    \item [(1)] Case $j=1$. Let $K \in \mathcal{H}_1$ be a conjugate of $N_G(M^m)$ in $G$, where $M$ is a maximal intransitive subgroup of $A_n$. Notice that $K \cap \Pi = \varnothing$ by Lemma \ref{pi_i}, because $N_{S_n}(M)$ is intransitive and hence it does not contain $n$-cycles. Therefore $f_{H}(K)=0$, implying that $P(E_{v,1})=0$.     
    
     \item [(2)] Case $j=2$. Assume $K$ is a maximal subgroup of $G$ conjugate to $N_G(M^m)$ where $M^m=K \cap N$, $M$ is the intersection between $A_n$ and a primitive maximal subgroup of $S_n$ distinct from $A_n$. Since $|M| \leqslant 4^n$ by \cite{primitivo}, $KN=G$ and $K \cap N$ is conjugate to $M^m$, we have $|C(H) \cap K| \leqslant |K|=2m \cdot |M|^m \leqslant 2m \cdot 4^{mn}$. Therefore, by Inequality (\ref{ineq_cvj}) and Lemmas \ref{cvj}, \ref{conjnm},
\begin{align*}
    P(E_{v, 2}) & \leqslant n \cdot mn \cdot \frac{mn \cdot 4^{mn}}{(n/(2e))^{mn}} 
    = m^2n^3 \cdot \left( \frac{8 e}{n} \right)^{n m} .
\end{align*}

    \item [(3)] Case $j=3$. Assume $K = N_G(M \times M^{a_2} \times \ldots \times M^{a_m})$, $M$ is a maximal imprimitive subgroup of $A_n$ with $3$ blocks, and $a_2, \ldots, a_m \in A_n$. We will bound the size of $C(H) \cap K$.
    Let $g \in C(H) \cap K = H \cap \Pi \cap K$, then $g = (x_1, \ldots, x_m)\gamma$, where $x_1 \ldots x_m\tau \in B$, and the fact that $g \in H \cap K$ implies that $x_1, \ldots, x_{m-1}, x_m \tau \in N_{S_n}(L)$, $a_i x_i a_{i+1}^{-1} \in N_{S_n}(M)$ for $i=1,\ldots,m-1$, where $a_1=1$, and $a_m x_m \tau \in N_{S_n}(M)$. We deduce that
$$\begin{array}{l}
x_1 \ldots x_i \in N_{S_n}(L) \cap N_{S_n}(M) a_{i+1} \hspace{1cm} \forall i=1,\ldots,m-1, \\
 x_1 \ldots x_m\tau \in B \cap N_{S_n}(L) \cap N_{S_n}(M)
 \end{array}$$
By induction, the number of choices for $x_i$ is $|N_{S_n}(L) \cap N_{S_n}(M)a_{i+1}|$, which is at most $|N_{S_n}(L) \cap N_{S_n}(M)|$, for every $i=1,\ldots,m-1$. Moreover, after choosing $x_1,\ldots,x_{m-1}$, the number of choices for $x_m$ is $|B \cap N_{S_n}(L) \cap N_{S_n}(M)|$, which is at most $|N_{S_n}(L) \cap N_{S_n}(M)|$. Therefore
\begin{align*}
|C(H) \cap K| \leqslant |N_{S_n}(L) \cap N_{S_n}(M)|^{m}.
\end{align*}
The above discussion implies that, if $B \cap N_{S_n}(L) \cap N_{S_n}(M)$ is empty, then $f_H(K)=0$, so now we may assume that there is an element $\sigma \in B \cap N_{S_n}(L) \cap N_{S_n}(M)$. Then $\sigma$ is an $n$-cycle normalizing $L$ and $M$. Let $\Delta$ and $\overline{\Delta}$ be the blocks of $L$, i.e. the two orbits of $\langle \sigma^2 \rangle$, and let $B_1$, $B_2$, $B_3$ be the blocks of $M$, i.e. the three orbits of $\langle \sigma^3 \rangle$. Then the six orbits of $\langle \sigma^6 \rangle$ are $\Delta \cap B_i$, $i=1,2,3$, and $\overline{\Delta} \cap B_i$, $i=1,2,3$, forming a partition $P$ of $\{1,\ldots,n\}$ consisting of $6$ blocks of size $n/6$. Clearly, $N_{S_n}(L) \cap N_{S_n}(M)$ is contained in the stabilizer of the partition $P$, which is isomorphic to $S_{n/6} \wr S_6$, hence
\begin{align*}
f_H(K) & = \frac{|C(H) \cap K|}{|C(H)|} \leqslant \frac{|N_{S_n}(L) \cap N_{S_n}(M)|^{m}}{|C(H)|}
\leqslant \frac{n}{2} \cdot \left( \frac{(n/6)!^6 \cdot 6!}{(n/2)!^2} \right)^m.
\end{align*}
An easy application of Stirling's inequalities shows that this is at most 
$n^{O(1)m} (1/3)^{n m}$.
By Inequality (\ref{ineq_cvj}) and Lemmas \ref{cvj}, \ref{conjnm}, the same bound holds for $P(E_{v,3})$.

    \item [(4)] Case $j=4$. Assume $K$ is a maximal subgroup of $G$ conjugate to $N_G(M^m)$ where $K \cap N = M^m$ and $M$ is a maximal imprimitive subgroup of $A_n$ with $4$ blocks. Since $KN=G$ and $K \cap N$ is conjugate to $M^m$, $|K| = 2m \cdot |M|^m$, hence an application of Stirling's inequalities gives
\begin{align*}
    f_H(K) & \leqslant \frac{|K|}{|C(H)|} = \frac{2m \cdot ((n/4)!^4 \cdot 4!)^m}{2/n \cdot (n/2)!^{2m}}
    \leqslant n^{O(1)m} \cdot \left( \frac{1}{2} \right)^{n m}.
\end{align*}
Therefore, by Inequality (\ref{ineq_svj}) and Lemma \ref{sv4}, $P(E_{v,4}) \leqslant n^{O(1)m} (1/4)^{n m}$.

 \item [(5)] Case $j=5$. Assume $K$ is a maximal subgroup of $G$ conjugate to $N_G(M^m)$ where $K \cap N = M^m$ and $M$ is a maximal imprimitive subgroup of $A_n$ with $5$ or more blocks. By \cite[Theorem 3]{Blackburn}, $|M| \leqslant n^{O(1)} \cdot (n/(5e))^{n}$, and since $|K| = 2m \cdot |M|^m$,
\begin{align*}
    f_H(K) &  \leqslant \frac{2m \cdot ((n/(5e))^{n} \cdot n^{O(1)})^m}{2/n \cdot (n/(2e))^{nm}} 
    \leqslant n^{O(1)m} \cdot \left( \frac{2}{5} \right)^{n m}.
\end{align*}
 
By Inequality (\ref{ineq_cvj}) and Lemmas \ref{cvj}, \ref{conjnm}, the same bound holds for $P(E_{v,5})$.
     
     \item [(6)] Case $j=6$. Assume $K = N_G(\Delta_{\varphi})$ is a maximal subgroup of $G$ of diagonal type, so that $|K|=2m \cdot |A_n|^{m/t}$ where $t$ is a prime divisor of $m$. Using $t \geqslant 2$ and Stirling's inequalities,
     \begin{align*}
    f_{H}(K) & \leqslant \frac{|K|}{|C(H)|} = \frac{2m (n!/2)^{m/t}}{(2/n) (n/2)!^{2m}}
     \leqslant n^{O(1)m} \cdot \left( \frac{2\sqrt{e}}{\sqrt{n}} \right)^{mn}.
\end{align*}
     By Inequality (\ref{ineq_cvj}) and Lemmas \ref{cvj}, \ref{conjnm}, the same bound holds for $P(E_{v,6})$.

\end{enumerate}

We now finish the proof by showing that $P(E_v) \leqslant \frac{1}{e(d+1)}$ for sufficiently large $n$. Recall that $d \leqslant 2^{mn}$. The above discussion implies that $P(E_{v,j}) \leqslant n^{O(1)m} (2/5)^{nm}$ for all $j=1,\ldots,6$, and since $P(E_v) \leqslant \sum_{j=1}^6 P(E_{v,j})$, it all boils down to showing that $n^{O(1)m} (2/5)^{mn} \leqslant (1/2)^{mn}$, which is true for large enough $n$.

\end{document}